\documentclass{amsart}
\usepackage{amsfonts,amssymb,amscd,amsmath,enumerate,verbatim,calc}
\usepackage{amsthm}

\newcommand{\ff}{\text{if and only if}}

\newcommand{\m}{\mathfrak{m} }
\newcommand{\M}{\mathfrak{M} }

\newcommand{\bbZ}{\mathbb{Z} }
\newcommand{\bbN}{\mathbb{N} }
\newcommand{\calR}{\mathcal{R} }
\newcommand{\calO}{\mathcal{O} }

\newcommand{\ord}{\operatorname{ord}}
\newcommand{\Div}{\operatorname{div}}

\newcommand{\depth}{\operatorname{depth}}

\newcommand{\Spec}{\operatorname{Spec}}
\newcommand{\Soc}{\operatorname{Soc}}

\newcommand{\LL}{\ell\ell}

\theoremstyle{plain}

\newtheorem{thm}{Theorem}[section]
\newtheorem{cor}[thm]{Corollary}
\newtheorem{lemma}[thm]{Lemma}
\newtheorem{prop}[thm]{Proposition}
\newtheorem{quest}[thm]{Question}
\newtheorem{conj}[thm]{Conjecture}

\theoremstyle{definition}
\newtheorem{defn}[thm]{Definition}
\newtheorem{rem}[thm]{Remark}
\newtheorem{exam}[thm]{Example}
\newtheorem{ac}{Acknowledgment}

\makeatletter
    
    \@addtoreset{equation}{section}
  \makeatother
\begin{document}
\title[The strong Rees property of powers of the maximal ideal]{The strong Rees property of powers of 
the maximal ideal and Takahashi-Dao's question}
\date{\today}
\author{Tony J.Puthenpurakal,  Kei-ichi Watanabe, 
Ken-ichi Yoshida}

\address{Tony J.Puthenpurakal,  Department of Mathematics, 
IIT Bombay, Powai, Mumbai 400076, India \\
email: tputhen@gmail.com}
%% Information for second author
\address{Kei-ichi Watanabe,  Department of Mathematics, 
College of Humanity and Sciences, Nihon University, 
Setagaya-ku, Tokyo, 156-8550, Japan\\
email: watanabe@math.chs.nihon-u.ac.jp}

%%%%%%%%%%%%%%%%%%%%%
%% Information for third author
%%%%%%% %%%%%%% %%%%%%% %%%%%%% %%%%%%% %%%%%%% %%%%%%%
\address{Ken-ichi Yoshida, 
 Department of Mathematics,
College of Humanities and Sciences,
Nihon University, Setagaya-ku, Tokyo, 156-8550, Japan\\
email: yoshida@math.chs.nihon-u.ac.jp}
\dedicatory{Dedicated to Craig Huneke on the occasion of 
his 65th Birthday.}

%\begin{dedication}
%Dedicated to Craig Huneke on the occasion of his 65th Birthday.
%\end{dedication}
\begin{abstract}
In this paper, we introduce the notion of the strong 
Rees property (SRP) for $\m$-primary ideals of a Noetherian local ring and
 prove that any power of the maximal ideal $\m$ has its property 
if the associated graded ring $G$ of $\m$ satisfies  
$\depth G \ge 2$. 
As its application, we characterize two-dimensional 
excellent normal local domains so that $\m$ is a $p_g$-ideal.  

Finally we ask what $\m$-primary ideals have SRP and state a conjecture which characterizes the case when $\m^n$ are the only ideals 
which have SRP.
\end{abstract}
\maketitle

\section{Introduction}
Let $(A,\m)$ be a Noetherian local ring with $d=\dim A \ge 1$ and $I$ an
$\m$-primary ideal of $A$.  
The notion of $\m$-full ideals was introduced by 
D. Rees and J. Watanabe  (\cite{JW}) and they proved the \lq\lq Rees property" 
for $\m$-full ideals, namely, if $I$ is $\m$-full ideal and $J$ is an 
ideal containing $I$, then $\mu(J)\le \mu(I)$, where $\mu(I)=\ell_A(I/\m I)$ 
is the minimal number of generators of $I$. Also, they proved that integrally closed 
ideals are $\m$-full if $A$ is normal. 

\par
Suppose $\depth A >0$. 
Then $\widetilde{\m^n}$, 
the Ratliff-Rush closure of $\m^n$, is $\m$-full (\cite[Proposition 2.2]{AP}). 
Thus $\m^n$ is $\m$-full for sufficiently large $n$. 
\par 
Sometimes we need stronger property for $\mu(I)$ and we will call it 
\lq\lq Strong Rees property" (SRP for short). 

\par \vspace{2mm} \par \noindent 
{\bf Definition (Strong Rees Property).}
Let $I$ be an $\m$-primary ideal of $A$. 
Then we say that $I$ satisfies the {\it strong Rees property} 
if for every ideal $J \supsetneq I$, we have $\mu(J) < \mu(I)$.

\par \vspace{2mm}
So it will be natural to ask the following questions.

\begin{quest}\label{Q_SRP} 
Let $(A,\m)$ be a normal local ring. 
Then 
\begin{enumerate}
\item[$(1)$] Does $\m^n$ has the strong Rees property for every $n\ge 1$?
\item[$(2)$] If $I$ has the strong Rees property, 
then is $I = \m^n$ for some $n \ge 1$?
\end{enumerate}
\end{quest}

\par
Actually, both questions are not true in general. 
But we can show that
(1) in the Question \ref{Q_SRP} holds under suitable mild condition.
Also, we will give an example of two-dimensional normal local rings 
where $\m^2$ does not satisfy the strong Rees property.

\par
As for (2) of Question \ref{Q_SRP}, we will 
discuss it in Section 6. 

\par \vspace{2mm}
Assume that $I$ is an $\m$-primary ideal. 
Then the \textit{multiplicity} 
(resp. \textit{the minimal number of system of generators}) 
of $I$ is denoted by $e(I)$ (resp. $\mu(I)$). 
The \textit{Loewy length} $\LL(I)$ is defined by 
\[
\LL(I)=\min\{r \in \mathbb{Z}_{\ge 0} \,|\,\m^r \subset I \}
\]
Notice that the notion of Loewy length of an Artinian ring 
measures the nilpotency of the maximal ideal. 
It is natural to ask if $e(I)$ is bounded by the product 
of $\mu(I)$ and $\LL(I)$ after adjusting some error terms.

\par 
The origin of this work was a discussion of second and 
third authors with Hailong Dao.  
He presented an inequality between $\mu(I)$, 
the multiplicity $e(I)$ of $I$ and Loewy length $\LL(I)$.

\par \vspace{2mm} \par \noindent
{\bf Question \ref{T-Dao} (\cite{DS}).}
Let $(A,\m)$ be a $d$-dimensional Cohen-Macaulay local ring. 
Let $I \subset A$ be an $\m$-primary $($integrally closed$)$ ideal.  
When does the inequality 
\[
(d-1)!(\mu(I)-d+1)\cdot \LL(I) \ge e(I)
\]
hold true?

\par
Dao and Smirnov \cite{DS} proved that the Question \ref{T-Dao} holds true 
if $A$ is a two-dimensional analytically unramified and 
the maximal ideal $\m$ is  a $p_g$-ideal (or, equivalently, 
the Rees algebra $\calR(\m)$ is normal and Cohen-Macaulay).

\par 
We are interested in the converse of the Question 
\ref{T-Dao} in the case of $d=2$. 
Here, since $e(I)$ does not change after taking integral closure, 
it is natural to assume that $I$ is integrally closed.
Namely, our Question is 

\begin{quest}\label{conv-Dao} 
Assume that $(A,\m)$ is a two-dimensional excellent normal local domain. 
If an inequality 
\[
(\mu(I)-1)\cdot \LL(I) \ge e(I) 
\] 
holds for any $\m$-primary integrally 
closed ideal $I$, then is $\m$ a $p_g$-ideal? 
\end{quest}

\par
It turns out that this question is related to the 
\lq\lq strong Rees property" of powers of the maximal ideal.
The main result in this paper is the following theorem.

\par \vspace{2mm} \par \noindent
{\bf Theorem \ref{Strong}.}
Let $(A,\m)$ be a Noetherian local ring. 
Assume that $\depth A \ge 2$ and 
$H_{\M}^1(G)$ has finite length, where $G=G(\m)= 
\oplus_{n\ge 0}\m^n/\m^{n+1}$.  
If $\m^{\ell}$ is Ratliff-Rush closed, then $\m^{\ell}$ has the strong Rees property. 

\par \vspace{2mm}
By Remark \ref{RR&G}, any power $\m^{\ell}$ 
is Ratliff-Rush closed whenever $\depth G \ge 1$. 
Hence we have the following corollary. 

\par \vspace{2mm} \par \noindent 
{\bf Corollary \ref{depth2}.}
If $\depth G \ge 2$, then $\m^{\ell}$ has the 
strong Rees property for every $\ell \ge 1$. 

\par 
In general, we cannot relax the assumption that 
$\depth G \ge 2$ even if $A$ is normal. 
See Section 4 for more details.

\par 
In Section 5, as an application of the theorem above, 
we prove the above question has an affirmative answer. 

\par \vspace{2mm} \par \noindent 
{\bf Theorem \ref{Dao}.}

Let $(A,\m)$ be a two-dimensional excellent normal local domain containing an algebraically closed field.
Then the following conditions are equivalent$:$
\begin{enumerate}
\item[$(1)$] $\m$ is a $p_g$-ideal. 
\item[$(2)$] For every $\m$-primary integrally closed ideal $I$, 
\begin{equation}
(\mu(I)-1) \cdot \ell\ell_A(I) \ge e(I) \tag{$*$}
\end{equation}
holds true. 
\item[$(3)$] For any power $I=\m^{\ell}$ of $\m$, 
the inequality $(*)$ holds true.   
\end{enumerate}

\par \vspace{2mm}
After proving this theorem, Dao and Smirnov informed 
us that they proved the same theorem independently.

%%%%%%%%%%%%%%%%%%%%%%%%%%%%%%%%%%%%%%%%%%
%%%%%%%%%%%%%%%%%%%%%%%%%%%%%%%%%%%%%%%%%%%
%%%%%%%%%%%%%%%%%%%%%%%%%%%%%%%%%%%%%%%%%%%%

\section{Preliminaries}
Throughout this paper, 
let $(A,\m)$ be a Noetherian local ring with $d=\dim A \ge 1$, 
and let $I,J \subset A$ be ideals of positive height. 
We recall the notion which we will need later. 

\subsection{$\m$-full ideals}

\begin{defn}[\textbf{$\m$-full, Rees  property}\textrm{\cite{JW}}] \label{m-full}
An ideal $I$ is called \textit{$\m$-full} if 
there exists an element $x \in \m$ so that 
$\m I \colon x = I$. 
\par 
An ideal $I$ is said to have \textit{the Rees property}  
if $\mu(J) \le \mu(I)$ for any ideal $J$ containing $I$.  
\end{defn}

\begin{prop}[\textrm{See \cite[Theorem 3]{JW}}] \label{ReesP}
Any $\m$-primary $\m$-full ideal has the Rees property. 
\end{prop}

\begin{proof}
See the proof of \cite[Lemma 2.2]{G}. 
\end{proof}

\par 
The following result is due to Rees in the case of normal integral domains. 

\begin{prop}[\textrm{See \cite[Theorem 5]{JW}}] \label{IntClos}
Any integrally closed ideal of positive height is $\m$-full. 
\end{prop}

\begin{proof}
See the proof of Theorem \cite[Theorem 2.4]{G}.
\end{proof}

\begin{defn}[\textbf{Ratliff-Rush closure}  \textrm{\cite{RR}}] \label{RRclos}
The \textit{Ratliff-Rush closure} of $I$ is defined by 
\[
\widetilde{I} = \bigcup_{n \ge 0} I^{n+1} \colon I^n.
\] 
The ideal $I$ is called \textit{Ratliff-Rush closed} 
if $\widetilde{I}=I$. 
\end{defn}

\begin{rem}\label{RR&G}
Note that $I \subset \widetilde{I} \subset \overline{I}$, 
where $\overline{I}$ is the integrally closure of $I$. 
If we put $G=G(\m)=\bigoplus_{n \ge 0} \m^n/\m^{n+1}$, then 
\[
H_{\M}^0(G) = \bigoplus_{n \ge 0} (\widetilde{\m^{n+1}}\cap \m^{n})/\m^{n+1}.
\]
\end{rem}

\begin{lemma} Assume that for some $z\in \m^{n-1}\setminus
\m^n$ we have $z\m\subset \m^{n-1}$.  Then, putting 
$I= (x) + \m^n$, $\mu(I) > \mu(\m^n)$, hence $\m^n$ 
does not have Rees property. 
\end{lemma}
\begin{proof} Let $\{y_1,\ldots, y_{\mu}\}$ be a minimal set 
of generators of $\m^n$.  Then we can see that 
$\{z, y_1,\ldots, y_{\mu}\}$ is a minimal set of generators of 
$I$.  Hence $\mu(I) = \mu(\m^n)+1$.
\end{proof}

\par
Any $\m$-full ideal has the Rees property. 
However, in order to prove our theorem, we need 
stronger inequalities.  

\par 
Recall the definition of the strong Rees property. 
\begin{defn} \label{SRP-defn}
An $\m$-primary ideal $I$ is said to have the  
\textit{strong Rees property} (SRP for short) if  
$\mu(J) < \mu(I)$ holds true for every ideal $J$ with 
$I \subsetneq J$.
\end{defn}

\par 
We can show the existence of $\m$-primary ideals with SRP 
by the following Lemma.

\begin{lemma} \label{ExSRP}
Let $(A,\m)$ be a Noetherian local ring. 
Fix an $\m$-primary ideal $I$ with $\mu(I)=n$.
\begin{enumerate}
\item Any maximal element of the set
 of $\m$-primary ideals 
\[\mathcal{I}= \{ J \subset A \;|\; J\supset I \;\text {and}\;  \mu(J) \ge n\} \]
 has the strong Rees property.
\item If, moreover, $I$ is $\m$-full, then there is   
 an ideal $I'\supset I$ with the strong Rees property and $\mu(I')=\mu(I)$.
\end{enumerate}
\end{lemma}

\begin{proof} (1) is obvious by the definition of $\mathcal{I}$. 
\par
(2) If $I'\supset I$  has SRP and $\mu(I')\ge \mu(I)$,  
we should have $\mu(I') = \mu(I)$ since $I$ is $\m$-full.  
\end{proof}

The next example gives a motivation for us to 
study the strong Rees property of powers 
of the maximal ideal.  
See also \cite[Theorem 5]{JW} and Section 6. 

\begin{exam}
\label{RLR-SRP}
Assume that $A$ is a two-dimensional regular local ring. 
Then for any $\m$-primary ideal $I$, 
the following conditions are equivalent. 
\begin{enumerate}
\item $I$ has the strong Rees property. 
\item $I=\m^n$ for some integer $n \ge 1$. 
\end{enumerate}
Indeed, assume that $I$ has the strong Rees property.  
If we put $n=\ord(I)=\max\{n \in \bbZ\,|\, I \subset \m^n \}$, then we can 
take an element $f \in I$ so that $\ord(f)=n$. 
Then since $I/(f)$ is an $\m/(f)$-primary ideal of 
$A/(f)$, we have $\mu(I/(f))\le n=e(\m/(f))$ 
because $A/(f)$ is a one-dimensional Cohen-Macaulay local ring. 
In particular, $\mu(I) \le n+1 = \mu(\m^{n})$. 
If $I \ne \m^{n}$, then $n+1=\mu(\m^n) < \mu(I)$ by
 the assumption (1). 
But this is a contradiction.   
\par 
Conversely, if $I\supsetneq \m^n$, then $\ord(I)<n$ and 
$\mu(I) \le \ord(I) + 1 < n+1=\mu(\m^n)$.
\end{exam}

\par \vspace{2mm}
\subsection{$p_g$-ideals}
In what follows, let $(A,\m)$ be a 
two-dimensional excellent normal local domain 
containing an algebraically closed field $k=A/\m$. 
Let $f \colon X \to \Spec A$ be a resolution of 
singularities. 
Then $p_g(A)=\dim_k H^1(\calO_X)$ is called  
the \textit{geometric genus} of $A$. 
Note that it does not depend on the choice of 
resolution of singularities.

\par \vspace{2mm}
Let $I \subset A$ be an $\m$-primary integrally closed 
ideal. 
Let $f \colon X \to \Spec A$ be a resolution of 
singularities on which $I$ is represented, that is, 
$I\calO_X$ is invertible and $I\calO_X=\calO_X(-Z)$ for 
some anti-nef cycle $Z$ on $X$. 
\par 
Okuma and the last two authors \cite{OWY1} proved that 
$\dim_k H^0(\calO_X(-Z)) \le p_g(A)$ holds true 
if $H^0(\calO_X(-Z))$ has no fixed component.

\begin{defn}[\textrm{See \cite{OWY1}}] \label{PG-def}
An anti-nef cycle $Z$ is a $p_g$-cycle if $\calO_X(-Z)$ is generated and 
$\dim_k H^0(\calO_X(-Z))=p_g(A)$.  
An $\m$-primary integrally closed ideal $I$ is called a \textit{$p_g$-ideal} 
if $I$ is represented 
as $I=H^0( \calO_X(-Z))$ by a $p_g$-cycle $Z$. 
\end{defn}

The following theorem gives a characterization 
of $p_g$-ideals in terms of Rees algebras.

\begin{prop}[See \cite{OWY2}] \label{pg-chara}
Let $(A,\m)$ be a two-dimensional excellent normal local domain 
containing an algebraically closed field. 
Let $I \subset A$ be an $\m$-primary ideal. 
Then the following conditions are equivalent$:$
\begin{enumerate}
\item[$(1)$] $I$ is a $p_g$-ideal. 
\item[$(2)$] $\overline{I^n}=I^n$ for every $n \ge 1$ and $I^2=QI$ 
for some minimal reduction $Q \subset I$. 
\item[$(3)$] The Rees algebra $\calR(I)=A[It] (\subset A[t])$ is a Cohen-Macaulay normal domain, where 
$t$ is an indeterminate over $A$.  
\end{enumerate} 
\end{prop}

For instance, any integrally closed $\m$-primary ideal 
in a two-dimensional rational singularity (i.e. $p_g(A)=0$) 
is a $p_g$-ideal. 
On the other hand, any two-dimensional excellent 
normal local domain containing algebraically closed field 
has a $p_g$-ideal; see \cite{OWY1,OWY2}.

%%%%%%%%%%%%%%%%%%%%%%%%%%%%%%%%%%%%%%%%%%%%%%
\section{Strong Rees Property of powers of the maximal ideal}

\par \vspace{2mm}
In what follows, let $(A,\m)$ be a Noetherian local ring and set $\mathcal{R}= \mathcal{R}(\m)$ and $G=G(\m)$ 
and $\mathfrak{M}=\m \mathcal{R}+\mathcal{R}_{+}$. 
The main purpose of this section is to consider the following question.

\begin{quest} \label{Motivation}
Assume that $I$ is Ratliff-Rush closed.  
When does $I$ have the strong Rees property? 
\end{quest}

As an answer, we show that some powers of the maximal 
ideal $\m$ have the strong Rees property
if $\depth G \ge 2$; see Corollary \ref{depth2}.  
More generally, we can show the following theorem.

\begin{thm}[\textbf{Strong Rees Property}] \label{Strong}
Assume that $\depth A \ge 2$ and 
$H_{\M}^1(G)$ has finite length. 
If $\m^{\ell}$ is Ratliff-Rush closed, then $\m^{\ell}$ has the strong Rees property. 
\end{thm}

\par \vspace{2mm}
We first need to prove the following lemma. 

\begin{lemma} \label{Full}
Suppose that 
 for some $x\in \m$, $\m^{\ell+1} \colon x = \m^{\ell}$. 
Then for every ideal $J$ with $J \supset \m^{\ell}$, 
$\mu(J) \le \mu(\m^{\ell})$ is always satisfied, and 
the following conditions are equivalent$:$
\begin{enumerate}
 \item[$(1)$] $\mu(J)=\mu(\m^{\ell})$. 
 \item[$(2)$] $\m J= xJ + \m^{\ell+1}$. 
\end{enumerate}
When this is the case, if we put 
$C=J\mathcal{R}/\m^{\ell}\mathcal{R}$, 
%we obtain the following exact sequence$:$ 
%\[
%0 \to [0 \colon xt](-1) \to C(-1)
%\stackrel{\cdot xt} {\longrightarrow}
%C \to {}_h \big[J/\m^{\ell}\big] \to 0. 
%\]
$C/(xt)C = (C/(xt)C)_0 = J/\m^{\ell}$. 
In particular, $\dim C \le 1$. 
\end{lemma}

\begin{proof}
Consider the following two short exact sequences$:$
\[
0 \; \to \;  \dfrac{\m^{\ell} \cap \m J}{\m^{\ell+1}} \; \to \; \dfrac{\m^{\ell}}{\m^{\ell+1}} 
\; \to \; \dfrac{\m^{\ell}}{\m^{\ell} \cap \m J} \; \to \;  0.
\]
\[
0 \; \to \; \dfrac{\m^{\ell}+\m J}{\m J} \; \to \;\dfrac{J}{\m J} 
\;\to \;\dfrac{J}{\m^{\ell}+\m J} \;\to \; 0.
\]
Since $\m^{\ell}/(\m^{\ell} \cap \m J) 
\cong (\m^{\ell}+\m J)/\m J$, combining two exact sequence 
implies 
\begin{equation}
0 \; \to \;  \dfrac{\m^{\ell} \cap \m J}{\m^{\ell+1}} \; \to \; \dfrac{\m^{\ell}}{\m^{\ell+1}} 
\; \to \;\dfrac{J}{\m J} 
\;\to \;\dfrac{J}{\m^{\ell}+\m J} \;\to \; 0.
\end{equation}
It follows that 
\[
\ell_A(\m^{\ell} \cap \m J/\m^{\ell+1})=
\big\{\mu(\m^{\ell})-\mu(J)\big\}+ 
\ell_A(J/\m^{\ell}+\m J).
\]
Furthermore, since $\m J/(\m^{\ell} \cap \m J) \cong (\m^{\ell}+\m J)/\m^{\ell}$, we get 
\[
\ell_A(\m J/\m^{\ell+1})=\ell_A(J/\m^{\ell})
+\big\{\mu(\m^{\ell})-\mu(J)\big\}
\]
\par 
We now consider an $\mathcal{R}$-module 
 $C=J\mathcal{R}/\m^{\ell}\mathcal{R}$. 
The assumption that $\m^{\ell+1}\colon x=\m^{\ell}$ implies that the multiplication map 
\[
\cdot x \colon C_0 = J/\m^{\ell} \; \to \; 
C_1 = \m J/\m^{\ell+1}
\]
is injective. 
Hence $\mu(\m^{\ell})-\mu(J)=\ell_A(C_1)-\ell_A(C_0) \ge 0$.  
Moreover, equality holds true if and only if  
the multiplication map by $x$ is isomorphism, 
which means that $\m J = xJ +\m^{\ell+1}$.   
\par 
When this is the case, we have 
\[
\m^{n+1} J= x \m^n J + \m^{n+\ell+1}
\]
for every $n \ge 1$. The last assertion immediately 
follows from this. 
\end{proof}

\begin{lemma} \label{Depth}
Let $J\subset A$ be an ideal with 
$J \supsetneq \m^{\ell}$, where $\ell \ge 1$.  
Put $C = J\mathcal{R}/\m^{\ell}\mathcal{R}$.  
Assume that $\depth A \ge 2$, and 
$H_{\M}^1(G)$ has finite length. 
If $\m^{\ell}$ is Ratliff-Rush closed, then 
\begin{enumerate}
\item[$(1)$] $H_{\M}^i(C)$ has finite length for $i=0,1$.
\item[$(2)$] $[H_{\M}^0(C)]_0=0$. 
\end{enumerate}
\end{lemma}

\begin{proof}
First we define an $\calR$-module $L(-1)$ as follows$:$
\[
0 \to \calR \to A[t] \to L(-1) \to 0 \;\text{(ex)},
\]
that is, $L(-1)=\bigoplus_{n \ge 0} (A/\m^n) t^n$. 

\begin{flushleft}
{\it Claim 1.} $H_{\M}^i(L(-1))$ has  finite length for 
$i=0,1$. 
\end{flushleft}

By \cite[Proposition 4.7]{P1} we have
\[
H^0_{\M}(L(-1)) = \bigoplus_{n \geq 0} \frac{\widetilde{\m^n}}{\m^n} =H^0_{\M}(G),
\]
which has finite length and then it is proved in 
\cite[Theorem 6.2]{P2} that $H_{\M}^1(L(-1))$ has finite length,
if and only if $H_{\M}^1(G)$ has finite length.
For instance, $\depth G \ge 2$, then  $H_{\M}^i(L(-1))=0$ for each $i=0,1$.

\par \vspace{2mm}
Secondly, we define $D=L(-1)_{\ge \ell}(\ell)
=\bigoplus_{n \ge 0} (A/\m^{\ell+n}) t^n$.   

\begin{flushleft}
{\it Claim 2.} $[H_{\M}^0(D)]_0=0$ and $H_{\M}^1(D)$ has finite length. 
\end{flushleft}
\par
By definition, we have 
\[
0 \;\to\; D(-\ell) \;\to \; L(-1) \;\to\; W \;\to\; 0, 
\]
where $W$ is an $\calR$-module of finite length. 
Then since in the exact sequence 
\[
W=H_{\M}^0(W) \to H_{\M}^1(D(-\ell)) \to H_{\M}^1(L(-1))
\]
the modules of the both sides 
have finite length, 
so does  $H_{\M}^1(D(-\ell))$. 
Moreover, as $H_{\M}^0(D(-\ell)) \subset H_{\M}^0(L(-1))$,
$H_{\M}^0(D(-\ell))$ is also of finite length. 
\par \vspace{2mm} 
The first assertion follows from the fact 
$[H_{\M}^0(D)]_0 \subset \widetilde{\m^{\ell}}/\m^{\ell}$
and our assumption.  

\par \vspace{2mm}
Thirdly, we define an $\calR$-module 
$V = \bigoplus_{n \ge 0} (A/\m^n J) t^n$ as follows$:$ 
\[
0 \; \to \;  J \calR \; \to \; A[t] \;\to\; V \;\to \; 0 
\;\text{(ex)}. 
\]

\par 
By definition of $V$, we have 
\[
0=H_{\M}^0(A[t]) \;\to\; H_{\M}^0(V) \;\to\;
H_{\M}^1(J\calR) \;\to\; H_{\M}^1(A[t])=0,  
\]  
where two vanishing follows from the fact that 
$\depth A \ge 2$. 
It follows that $[H_{\M}^0(V)]_n \cong [H_{\M}^1(J\calR)]_n =0$ 
for large enough $n$. 
On the other hand, as $H_{\M}^0(V) \subset V$, 
$[H_{\M}^0(V)]_n=0$ for each $n \le -1$. 
Thus $H_{\M}^0(V)$ has finite length. 

\par \vspace{2mm}
\begin{flushleft}
{\it Claim 3.} $[H_{\M}^0(C)]_0=0$ and $H_{\M}^1(C)$ has finite length. 
\end{flushleft}
\par \vspace{2mm}
One can easily obtain the following exact sequence:
\[
0 \; \to \;  C \; \to \; D \;\to\; V \;\to \; 0 
\;\text{(ex)}. 
\]
Hence 
\[
\begin{array}{ccccccc}
0 & \to & H_{\M}^0(C) & \to & H_{\M}^0(D) 
  & \to & H_{\M}^0(V) \\[2mm]
  & \to & H_{\M}^1(C) & \to & H_{\M}^1(D) 
  & \to & \cdots 
\end{array}
\]
As $[H_{\M}^0(D)]_0=0$ by Claim 2, we get $[H_{\M}^0(C)]_0=0$. 
Moreover, in the sequence 
\[
H_{\M}^0(D) \to H_{\M}^0(V) \to H_{\M}^1(C) \to 
H_{\M}^1(D),
\]
the both sides of $H_{\M}^1(C)$ have finite length. 
Hence so does $H_{\M}^1(C)$.   
\end{proof}

\begin{proof}[{\bf Proof of Theorem $\ref{Strong}$}]
Choose an $\m$-superficial element $x \in \m \setminus \m^2$. 
By assumption, we have 
\[  
\m^{\ell} \subset \m^{\ell+1} \colon x \subset 
\widetilde{\m^{\ell}}=\m^{\ell}.
\]
Then $\m^{\ell+1} \colon x =\m^{\ell}$
In particular, this means $\m^{\ell}$ is $\m$-full. 
\par
Let $J$ be an ideal with $J \supsetneq \m^{\ell}$. 
By Lemma \ref{Full}, $\mu(J) \le \mu(\m^{\ell})$. 
We want to show that this inequality is strict. 
Now suppose that equality holds true$:$ $\mu(J)=\mu(\m^{\ell})$. 
Put $C=J\calR/\m^{\ell}\calR$. 
In Lemma \ref{Full}, 
we showed that  $C/(xt)C$ has finite length.
Hence we have $\dim C \le 1$. 
\par 
On the other hand, by Lemma \ref{Depth}, we have 
$[H_{\M}^0(C)]_0=0$ and $H_{\M}^1(C)$ has finite length. 
If $\dim C=0$, then $0 \ne J/\m^{\ell} = [C]_0=[H_{\M}^0(C)]_0=0$. This is a contradiction. 
Hence $\dim C=1$. 
Then $H_{\M}^1(C)$ is \textit{not} finitely generated. 
This contradicts the fact that $H_{\M}^1(C)$ has finite length.  
Therefore we conclude that $\mu(J) < \mu(\m^{\ell})$, as required. 
\end{proof}

\par 
In the case where $\depth G \ge 2$, then 
all powers of the maximal ideal have the strong 
Rees property. 

\begin{cor} \label{depth2}
If $\depth G \ge 2$, then $\m^{\ell}$ has the 
strong Rees property for every $\ell \ge 1$. 
\end{cor}

\begin{proof}
Assume Theorem \ref{Strong} holds true. 
If $\depth G \ge 2$, then $H_{\M}^1(G)=0$ and 
$\depth A \ge 2$. 
Hence the ring $A$ satisfies the assumption on Theorem \ref{Strong}. 
\end{proof}

\begin{cor} \label{grCM}
Let $(A,\m)$ be a Cohen-Macaulay local ring with 
$d=\dim A \ge 2$. 
If $G=G(\m)$ is Cohen-Macaulay, then $\m^{\ell}$ has 
the strong Rees property for each $\ell \ge 1$. 
\end{cor}

\begin{cor}\label{normalRees}
If  $\depth A  \geq 2$ and $\m$ is a normal ideal
$($i.e., $\m^n$ is integrally closed for all $n \geq 1)$
 then $\m^{\ell}$ has the 
strong Rees property for every $\ell \ge 1$. 
\end{cor}
\begin{proof} By  \cite[Theorem 3.1]{HH} 
we get $\depth G(\m^n) \geq 2$ for all $n \gg 0$.  
Here  $G(\m^n)$ denotes the associated graded ring of $\m^n$. 
It follows that $H^1_{\M}(G(\m))$ has finite length. Also as $\m^r$ is integrally closed  for all $r \geq 1$ we get that it is Ratliff-Rush closed. It follows that $\m^r$ has SRP for all $r \geq 1$ by Theorem \ref{Strong}.
\end{proof}

\begin{exam} \label{SRP-exam}
Let $(A,\m)$ be a two-dimensional excellent normal 
local domain. 
If $\m$ is a $p_g$-ideal (e.g. $A$ is a rational singularity), then $\m^{\ell}$ has 
the strong Rees property for every $\ell \ge 1$. 
\end{exam}

\par 
On the other hand, \lq\lq$\depth G \ge 1$'' can be characterized by the Rees property. 

\begin{prop} \label{Rees}
Put $G =G(\m)$, where $d=\dim G \ge 1$.  
Then the following conditions are equivalent$:$
\begin{enumerate}
 \item[$(1)$] $\depth G \ge 1$. 
 \item[$(2)$] $\m^{\ell}$ has the Rees property for every $\ell \ge 1$.  
\end{enumerate}
\end{prop} 

\begin{proof}
$(1) \Longrightarrow (2):$ 
If $\depth G \ge 1$, then $\widetilde{\m^{\ell}}= \m^{\ell}$ and $\m^{\ell}$ is $\m$-full for every $\ell \ge 1$. 
In particular, $\m^{\ell}$ has the Rees property. 
\par \vspace{2mm} \par \noindent 
$(2) \Longrightarrow (1):$ 
Now suppose $\depth G=0$. Then $H_{\M}^0(G) \ne 0$.
There exist an integer $\ell \ge 1$ and an element 
$z \in \m^{\ell} \setminus \m^{\ell+1}$ such that 
$0 \ne z^{*}=z+\m^{\ell+1} \in [H_{\M}^0(G)]_{\ell} \cap \Soc(H_{\M}^0(G))$. 
Then $\m z \subset \m^{\ell+2}$. 
If we put $J=(z)+\m^{\ell+1}$, then 
$\m J= \m (z)+\m^{\ell+2}=\m^{\ell+2}$. 
Hence 
\[
J/\m J = ((z)+\m^{\ell+1})/\m^{\ell+2} \supsetneq 
\m^{\ell+1}/\m^{\ell+2}
\]
and so $\mu(J)=\ell_A(J/\m J)> \ell_A(\m^{\ell+1}/\m^{\ell+2})=\mu(\m^{\ell+1})$. 
This contradicts the assumption that $\m^{\ell+1}$ has the 
Rees property.
\end{proof}

\par \vspace{2mm}
We now consider the case of $\depth G=1$. 
We need the following lemma. 

\begin{lemma} \label{Socle-Ini}
Suppose that $\depth G=1$. 
Assume that $J \supsetneq \m^{\ell}$ so that $\mu(J)=\mu(\m^{\ell})$.  
Then one can find elements $x \in \m\setminus \m^2$ 
and $y \in J \setminus (\m^{\ell}+(x))$ such that 
$x^{*}$ is $G$-regular and $0 \ne z^{*} \in 
\Soc(H_{\M}^0(\overline{G}))$,
where $z=\overline{y}\in A/xA$ and 
$z^{*}$ denotes the initial form of $z$ in 
$\overline{G} =G/x^{*}G$.
\end{lemma}

%%% Proof of Lemma 3.10
\begin{proof}
Take an element $x \in \m \setminus \m^2$ so that 
$x^{*}$ is $G$-regular and $J \not \subset \m^{\ell}+(x)$;
see the remark below for the exitence of $x$. 
Choose $y \in J \setminus (\m^{\ell}+(x))$. 
Then $y \in (\m^{k}+(x)) \setminus (\m^{k+1}+(x))$ for some $k$ 
with $0 \le k \le \ell-1$. 
By assumption and Lemma \ref{Full}, we have that 
$\m y \subset \m J =xJ+\m^{\ell+1}$. 
Let $\overline{~\quad~}$ denote the image of the surjection $A \to A/xA$ and put $z = \overline{y}$. 
Then $\overline{\m} z \subset \overline{m}^{\ell+1} 
\subset \overline{m}^{k+2}$, which means that 
$0 \ne z^{*} \in [\Soc(H_{\M}^0(G/x^{*}G))]_k$. 
\end{proof}

\begin{rem} \label{exis-x}
Suppose $\depth G=1$. 
For any ideal $J \supsetneq \m^{\ell}$, one can find an element 
$x \in \m \setminus \m^2$ so that $x^{*}$ is $G$-regular
and $J$  is not contained in $\m^{\ell}+ (x)$.  
In order to prove this, it suffices to consider the case 
 $\ell(J/\m^{\ell})=1$, that is, 
$J =(f)+\m^{\ell}$, where 
$f\not\in \m^{\ell}$ and $\m f \subset \m^{\ell}$. 
\par 
If $J \subset \m^{\ell}+(x)$ for such an element $x$ as above, 
then $f -ax \in \m^{\ell}$ for some $a \in A$. 
As $x^{*}$ is $G$-regular (and thus $a^{*}x^{*} \ne 0$),  
we obtain the relation $a^{*}x^{*}=f^{*}$ in $G$.  
Let $G=k[X]/\mathfrak{a}$ and let $F$ be the inverse image of 
$f^{*}$ in the polynomial ring $k[X]$. 
As $k$ is an infinite field and $\dim k[X]/(\mathfrak{a}+(F)) \ge 1$, 
one can find a homogeneous element $X$ of degree one which 
does not vanish on $V(\mathfrak{a}+(F))$. 
The required assertion follows from here. 
\end{rem}

\par 
Hence we have the following. 
%%% Proposition 3.11
\begin{prop} \label{Upper}
Suppose that $\depth G=1$. 
Let $x^{*} \in G_1$ be a nonzero divisor of $G$ and 
put $\overline{G}=G/x^{*}G$. 
If $[\Soc(H_{\M}^0(\overline{G})]_{\ell} =0$ for all $\ell <n$, then $\m^{\ell}$ has the strong Rees property for all $\ell \le n$.   
\end{prop}

%%% Proof of Proposition 3.11
\begin{proof}
First we note that $\Soc(H_{\M}^0(G/x^{*}G))$ is independent of 
the choice of $x$.  In fact, the short exact sequence 
\[
0 \longrightarrow  G(-1)  \stackrel{x^{*}}{\longrightarrow} G\longrightarrow \overline{G}:=
G/x^{*}G  \longrightarrow 0
\]
yields a short exact seuence
\[
H_{\M}^0(G)=0 \longrightarrow H_{\M}^0(\overline{G})
 \longrightarrow H_{\M}^1(G)(-1) 
 \stackrel{x^{*}}{\longrightarrow} H_{\M}^1(G).
\]
Taking a socle, we get 
\[
0 \longrightarrow \Soc(H_{\M}^0(\overline{G}))
 \longrightarrow \Soc(H_{\M}^1(G))(-1) 
 \stackrel{x^{*}}{\longrightarrow} \Soc(H_{\M}^1(G)).
\]
Since the last map is a zero map, $\Soc(H_{\M}^0(\overline{G}))
\cong  \Soc(H_{\M}^1(G))(-1) $ is independent of the choice of $x$. 
\par 
Now suppose that $\m^{\ell}$ does not have SRP. 
Then we can find an ideal $J \supsetneq \m^{\ell}$ so that 
$\mu(J)=\mu(\m^{\ell})$.  
By Lemma \ref{Socle-Ini}, there exist an element $x \in \m \setminus \m^2$ such that $x^{*}$ is $G$-regular and 
$0 \ne [\Soc(H_{\M}^0(G/x^{*}G)]_k$ for some $0 \le k \le \ell-1 
\le n-1$. This contradicts the assumption. 
\end{proof}

%Yos0802
%\begin{prop}\label{Upper}
%Let $x^{*} \in G_1$ be a nonzero divisor of $G$ and 
%put $\overline{G}=G/x^{*}G$. 
%If $H_{\M}^0(\overline{G})_{\ell} =0$ for all $\ell <n$, then %$\m^{\ell}$ has the strong Rees property for all $\ell \le n$.   
%\end{prop}
%\begin{proof} Take $J\supsetneq \m^{\ell}$ and let $k$ be
% maximal so that $J\subset \m^k$. Take $y\in J, 
% y\not\in \m^{k+1}$.
%Then by Lemma \ref{Full}, we have 
%$\m J = x J + \m^{\ell+2} \subset x\m^k + \m^{k+2}$, 
%which implies that $y + \m^{k+1} \in H_{\M}^0(\overline{G})_{k}$. 
%Hence we must have $\ell \ge n$.
%\end{proof}

%%% Propostion 3.12
\begin{prop} \label{Border}
When $\depth G=1$, there exists an $n \in \mathbb{N} \cup \{\infty\}$ such that $\m^{\ell}$ has the strong Rees property 
if and only if $1 \le \ell \le n$. 
\end{prop}
 
\par 
In order to prove the proposition, it suffices to show 
the following lemma. 

%%% Lemma 3.13
\begin{lemma} \label{not-SRP}
Suppose that $\depth G=1$. 
If $\m^{\ell}$ does not have the strong Rees property, 
then neither does $\m^{\ell+1}$. 
\end{lemma}

%%% Proof of Lemma 3.13
\begin{proof}
Since $\depth G=1$, there exists an element $x \in \m \setminus \m^2$ so that $x^{*}=x + \m^2$ is $G$-regular. 
In particular, $\m^{\ell}$ and $\m^{\ell+1}$ are $\m$-full. 
By assumption and Lemma \ref{Full}, we can take an ideal 
$I \supsetneq \m^{\ell}$ so that $\m I = xI + \m^{\ell+1}$. 
Put $J = xI+\m^{\ell+1}$. Then 
\[
\m J = \m (x I + \m^{\ell+1})= \m (xI)+\m^{\ell+2}.  
\]
Moreover, we suppose that $J=\m^{\ell+1}$. 
Then $x I \subset \m^{\ell+1}$ and thus 
$I \subset \m^{\ell+1}\colon x=\m^{\ell}$. 
This contradicts the choice of $I$. 
Hence $J \supsetneq \m^{\ell+1}$ and $\mu(J)=\mu(\m^{\ell+1})$. 
This implies that $\m^{\ell+1}$ does not have SRP. 
\end{proof}

%%% Example 3.14
\begin{exam} \label{triple}
For any $n \ge 1$, there exists a triple 
$(a,b,c) \in \mathbb{N}^3$ such that 
$A=k[[s,t^a,t^b,t^c]]$ is a 
two-dimensional Cohen-Macaulay local domain such 
that $\m^{\ell}$ has the strong Rees property if and only if $1 \le \ell \le n$. 
\end{exam}

\begin{proof}
For a given $n \ge 1$, we can choose an integer 
$a=10^N > 2n$. 
Set $b=a+1$ and $c=(a-1)(a+1)-an$. 
Then $s,\,t^a,\,t^b,\,t^c$ is a minimal system of generators because $ab-a-b=(a-1)(a+1)-a \ge c (> b >a)$. 
\par 
Since 
\[
A=k[[s,x,y,z]]/(yz-x^{a+1-n},\,x^nz-y^{a-1},\,z^2-x^{a+1-2n}y^{a-2}),
\]
we get
\[
G=G(\m) \cong k[S,X,Y,Z]/(YZ,X^nZ,Z^2,Y^a). 
\]
Then $S$ is an $G$-regular and 
$\Soc(H_{\M}^0(G/SG))$ is generated by 
$\overline{X^{n-1}Z} \in [G/SG]_{n}$. 
Hence $\m^{\ell}$ has SRP if 
$1 \le \ell \le n$. 
\par 
If we put $I=(x^{n-1}z)+\m^{n+1} \supsetneq \m^{n+1}$, 
then $x \cdot x^{n-1}z = y^{a-1} \in 
\m^{a-1} \subseteq \m^{n+2}$. 
Similarly, we have that $y \cdot x^{n-1}z \in \m^{a} \subseteq \m^{n+2}$ and $z \cdot x^{n-1}z \in \m^{2a-n-2} \subset \m^{n+2}$.
Hence $\m I =(s \cdot x^{n-1}z)+\m^{n+2}$, and this implies 
that $\m^{n+1}$ does not have SRP. 
\end{proof}

\begin{exam} \label{semi-concrete}
Let $A=k[[s,t^4,t^5,t^{11}]]$ and $\m=(s,t^4,t^5,t^{11})$. 
Then $\m^2$ does \textit{not} have strong Rees property. 
In fact, $\m^2=(s^2,st^4,st^5,st^{11},t^8,t^9,t^{10})$ 
and thus $\mu(\m^2)=7$. 
If we put $I=(s^2,st^4,st^5,t^8,t^9,t^{10},t^{11}) = \m^2+(t^{11})$, then $I \supsetneq \m^2$ and $\mu(I)=7=\mu(\m^2)$. 
\par 
On the other hand, since $G \cong k[S,X,Y,Z]/(XZ,YZ,Z^2,Y^4)$ (see e.g. \cite[Section 2]{Sa3}), we have $\depth G=1$ and 
thus $\m^2=\widetilde{\m^2}$ is $\m$-full. 
Moreover, since $t^{11} \in \overline{\m^2} \setminus \m^2$, 
$\m^2$ is \textit{not} integrally closed.    
\end{exam}

\par
We can find an example of two-dimensional excellent normal 
local domains $(A,\m)$ for which $\m^2$ does not satisfy 
the strong Rees property and $A/sA \cong k[[t^4,t^5,t^{11}]]$ for some nonzero divisor $s$ of $A$. 
See the next section. 

%%%%%%%%%%%%%%%%%%%%%%%%%%%%%%%%%%%%%%%%%%%%%%%
\section{Point divisor on a smooth curve -- An example $\m^n$ does not have 
strong Rees property}

In this section we treat a class of normal graded rings of dimension $2$ 
and discuss whether $\m^n$ has the strong Rees property in such rings.

\begin{defn}\label{R_CPdefn} 
 Let $k$ be an algebraically closed field and $C$ be a smooth connected projective curve of genus $g$ over $k$.  We take a point $P\in C$ and 
define 
\[ 
H = H_{C,P} = \{ n\in \bbZ\;|\; h^0(C, \calO_C(nP))> h^0(C, \calO_C((n-1)P))\},
\]
where $h^i(C,\mathcal{F})=\dim_k H^i(C,\mathcal{F})$. 
It is easy to see that $H_{C,P}$ is an additive semigroup and $\bbN \setminus H_{C,P}$ has just $g$ elements. 
\par 
We define 
\[ 
R= R_H = R_{C,P} 
= \oplus_{n\ge 0} H^0(C, \calO_C(nP) T^n,
\]
as a subring of $k(C)[T]$, where $H^0(C,\calO_C(nP)
=\{ f\in k(C)\;|\; \Div_C(f)+nP\ge 0\}$.  Namely, $f \in H^0(C,\calO_C(nP)$, 
\ff $f$ has pole of order at most $n$ at $P$ and no other poles.
\end{defn}

\par %\vspace{2mm}
Then $R=R_{C,P}$ is a normal graded ring of dimension $2$ as treated in \cite{GW}, Chapter 5, \S 2.  
In the following, we fix $H=H_{C,P}$ and write $A=k[H]$ and $R= R_H$ so that $R/TR \cong A$, where $T = 1.T\in R_1$.
\par 
We write $H=\langle n_1,\ldots , n_e\rangle$ if 
$H= \{ \sum_{i=1}^e a_in_i\;|\; a_i \in \bbZ_{\ge 0}\; (i=1,\ldots, e)\}$.
In this case, we say that $H$ is generated by $e$ elements. 
We denote by $H_+$ the set of positive elements of $H$ and 
denote $n\in rH_+$ if $n= h_1+\ldots +h_r$ with $h_i\in H_+$ ($i=1,\ldots , r$).
 
\begin{rem}  
Given a semigroup $H$, sometimes there does not exist  
the pair $(C,P)$ such that $H_{C,P} = H$. 
But at least we know the existence 
of $(C,P)$ such that $H=H_{C,P}$ in the following cases (cf. \cite{Ko}, \cite{K-O}) ;
\begin{enumerate}
\item[$(1)$] $k[H]$ is a complete intersection.  
\item[$(2)$] $H$ is generated by $3$ elements.
\item[$(3)$] $H$ is generated by $4$ elements and $H$ is symmetric or pseudo-symmetric. 
\item[$(4)$] $g(H)\le 9$, where $g(H)$ is the number of positive integers not in $H$. 
\end{enumerate}  
\end{rem}    

\par  
We summarize some property of $R=R_{C,P}$. 
We put $\m = R_+$.

\begin{prop}\label{R_CP} 
Let $R=R_{C,P}$. 
An element of $R_n$ is denoted by $fT^n$, 
where $f\in k(C)$.  
We denote by $v(f)$ the order of the pole of $f$ at $P$. For non zero elements 
$f,g\in k(C)$, $v(fg)= v(f)+v(g)$.
\begin{enumerate}
\item $fT^n\in R_n$ if and only if $v(f)\le n$ and $f$ has no other poles on $C$.
\item Hence if $v(f)<n$, then $f T^n \in T^{n-v(f)} R$, because 
$f T^{v(f)}\in R_{v(f)}$.
\item If $H_{C,P} = \langle n_1,\ldots , n_e\rangle$, which are minimal generating system, then there are
elements $f_1,\ldots , f_e\in k(C)$ with $v(f_i) = n_i$
$(i=1,\ldots, e)$ such that 
$R=k[T, f_1T^{n_1},\ldots, f_eT^{n_e}]$.
\item  If $fT^n \in R_n$ and $v(f)=n$, 
then $fT^n\in \m^r$ if and only if $n\in rH_+$.  
\item $T\in R_1$ is a super regular element of $R$. Namely, 
if $T x\in \m^r$ for some $x\in R$, then $x\in \m^{r-1}$.
\end{enumerate}
\end{prop}

\begin{thm} \label{point-SRP}
Let $R=R_{C,P}$ and $H=H_{C,P}=\langle n_1,\ldots , n_e\rangle$. 
If for some $n \in rH_+, n\not\in (r-1)H_+$, 
$n+n_i \in (r+2)H_+$ for $i=1,\ldots, e$, 
then $\m^{n+1}$ does not have the strong Rees property.
 \end{thm}

\begin{proof} By the assumption, there is 
some $n\in (r-1)H_+, n\not\in rH_+$
such that $n + H_+ \subset (r+1)H_+$.  
Then we can take $f\in k(C)$ so that 
 $v(f) = n$ and $fT^n\in R_n$.  
Since $fT^n\not\in TR$, $fT^n \in \m^{r-1}$ 
and $fT^n\not\in \m^{r+1}$.  
We put $I = (\m^r, fT^n)$ and show that 
$\mu(I) = \mu(\m^r)$.  
Now, let homogeneous minimal generators of $\m$ be 
$\{T, g_1T^{n_1},\ldots , g_eT^{n_e}\}$.   
Then  among the homogeneous minimal generators 
of $\m (fT^n)$,  
$(fT^n)(g_iT^{n_i})\in \m^{r+1}$ by our 
assumption. 
Hence we can obtain minimal generating system of $I$ 
from that of $\m^r$, 
interchanging $fT^n$ and $fT^{n+1}$, obtaining 
$\mu(I) = \mu(\m^r)$.  
\end{proof}
 
\begin{cor}\label{SR_m^n}   
Let $R=R_{C,P}$,  $H=H_{C,P}$ and $\m= R_+$.  
Then the following conditions are equivalent$:$
\begin{enumerate}
 \item For all $n\ge 2$, $\m^n$ has the 
strong Rees property.
\item The associated graded ring of $k[H]$ 
with respect to $k[H]_{+}$ is Cohen-Macaulay.
\item The associated graded ring of $R$ with respect to $R_+$ is Cohen-Macaulay.  
\end{enumerate} 
\end{cor}

\begin{exam}  \label{triple-notSRP}
Let $H=\langle 4,5,11\rangle$, 
$C$ a smooth curve of genus $5$ such that 
there is a point $P$ with 
$H_{C,P} = \langle 4,5,11\rangle$.
We put $R= R_{C,P}$ and $\m= R_{+}$.  
Since $11+4 \in 3H_+$ and $11 + 5 \in 4 H_+$, 
we see that $\m^2$ does \textit{not} have the strong Rees property. 
In this example, we can easily see that 
$\m^n$ is integrally closed for all $n\ge 2$.
\end{exam} 

\begin{rem} \label{semi-CM}
For $3$ generated semigroup 
$H=\langle a,b,c \rangle$,
we know when the associated graded ring is Cohen-Macaulay
(cf, \cite{He},\cite{LV}).
\end{rem}
%%%%%%%%%%%%%%%%%%%%%%%%%%%%%%%%%%%%%%%%%%%%%%%%%%%
%%%%%%%%%%%%%%%%%%%%%%%%%%%%%%%%%%%%%%%%%%%%%%%%%%%
\section{Takahashi-Dao's question}

Dao and Takahashi \cite{DT} gave two upper bounds of 
the dimension of the singularity category 
$\dim D_{\mathrm{sg}}(A)$:
\begin{eqnarray*}
\dim D_{\mathrm{sg}}(A) & \le & (\mu(I)-\dim A+1)\ell\ell(I)-1 \\[2mm]
\dim D_{\mathrm{sg}}(A) &\le & e(I)-1
\end{eqnarray*}
for any $\m$-primary ideal $I$ contained in the sum $\mathcal{N}^A$ of the 
Noether differents of $A$. 
They posed the following question. 

\begin{quest}[\textrm{Takahashi-Dao}] \label{T-Dao}
Let $(A,\m)$ be a $d$-dimensional Cohen-Macaulay local ring. 
Let $I \subset A$ be an $\m$-primary 
$($integrally closed$)$ ideal.  
When does an inequality 
\[
(d-1)!(\mu(I)-d+1)\cdot \ell\ell_A(I) \ge e(I)
\]
hold true $($cf. \cite{DS}$)$? 
\end{quest}

\par 
The following theorem is motivated by the question as above. 
In fact, $(1) \Rightarrow (2)$ is due to Dao and Smirnov, and $(2) \Rightarrow (1)$ is also proved by them independently.

\begin{thm} \label{Dao}
Let $(A,\m)$ be a two-dimensional excellent normal local domain containing $k=\overline{k}\cong A/\m$.
Then the following conditions are equivalent$:$
\begin{enumerate}
\item[$(1)$] $\m$ is a $p_g$-ideal. 
\item[$(2)$] For every $\m$-primary integrally closed ideal $I$, 
\begin{equation}
(\mu(I)-1) \cdot \ell\ell_A(I) \ge e(I) \tag{$*$}
\end{equation}
holds true. 
\item[$(3)$] For any power $I=\m^{\ell}$ of $\m$, the inequality $(*)$ 
holds true.   
\end{enumerate}
\end{thm}

\begin{proof}
$(1) \Longrightarrow (2):$ 
We give a sketch of proof here for the sake of the completeness. 
Assume that $\m$ is a $p_g$-ideal. 
Let $I \subset A$ be an $\m$-primary integrally closed ideal. 
Then there exists a resolution of singularities 
$f \colon X \to \Spec A$ so that $I\mathcal{O}_X = \mathcal{O}_X(-Z)$ for some anti-nef cycle $Z$ on $X$. 
By \cite[Theorem 6.1]{OWY1}, we have 
\[
\mu(I)=-MZ+1,
\]
where $M$ is an anti-nef cycle on $X$ so that $\m \mathcal{O}_X=\mathcal{O}_X(-M)$. 
\par 
Put $r=\ell\ell_A(I)$, that is, $\m^r \subset I$ and $\m^{r-1} \not \subset I$. 
Then $\overline{\m^{r}} \subset \overline{I}=I$. Thus $rM \ge Z$. 
Since $e(I)=-Z^2$, we have 
\[
(\mu(I)-1)\ell\ell_A(I) -e(I)=(-MZ)r +Z^2 = -(rM-Z)Z \ge 0,
\]
as required. 
\par \vspace{2mm} \par \noindent
$(2) \Longrightarrow (3):$ Trivial. 
\par \vspace{2mm} \par \noindent
$(3) \Longrightarrow (1):$ Now assume that inequalities 
\[
(\mu(\overline{\m^{\ell}})-1)\cdot
\ell\ell_A(\overline{\m^{\ell}})
\ge e(\overline{\m^{\ell}})
\]
hold true for all integers $\ell \ge 1$.
This shows that 
\[
(\mu(\overline{\m^{\ell}})-1)\cdot \ell \ge 
e(\overline{\m^{\ell}})=e(\m^{\ell})=\ell^2 \cdot e,
\]
where $e=e(\m)$. 
Hence $\mu(\overline{\m^{\ell}})\ge \ell e +1$. 
\par 
In the case of $\ell=1$, we have that 
$\mu(\m)-1 \ge e(\m)$. 
On the other hand, Abhyankar's inequality implies that 
$\mu(\m)-1 \le e(\m)$, and thus equality holds true.  
That is, $A$ has maximal embedding dimension in the sense 
of Sally (\cite{Sa2}). 
Then $\mu(\m^{\ell})= \ell e +1$. 
Moreover, since $G=G(\m)$ is Cohen-Macaulay (\cite{Sa1}), 
we obtain that $\m^{\ell}=\widetilde{\m^{\ell}}$ is 
$\m$-full for every $\ell \ge 1$. 
Then $\mu(\overline{\m^{\ell}}) \le \mu(\m^{\ell})=\ell e+1$
by Rees property of $\m$-full ideals. 
%; see \cite[Theorem 3]{JW}. 
Hence $\mu(\overline{\m^{\ell}})=\mu(\m^{\ell})$. 
Then Corollary  \ref{grCM} yields that 
$\overline{\m^{\ell}}=\m^{\ell}$ because $G(\m)$ 
is Cohen-Macaulay in our case (Sally \cite{Sa1}). 
Therefore $\m$ is a $p_g$-ideal by \cite{OWY2}. 
\end{proof}

\begin{rem}
We can show that $I=\m^{\ell}$ satisfies the above inequality if $(A,\m)$ has maximal embedding dimension. 
Note that $\m^{\ell}$ does \textit{not} necessarily have 
the strong Rees property; see Sections 3 and 4. 
\end{rem} 

\par \vspace{2mm}
It is known that the maximal ideal $\m$  
of any two-dimensional rational singularity 
is a $p_g$-ideal. 
So it is natural to ask the following question. 
Which ring the maximal ideal of which is a $p_g$-ideal?
By a similar argument as in \cite[Corollary 11.4]{GTT}, 
we can show the following. 
Notice that this gives a slight generalization of the fact 
that any two-dimensional rational singularity is an 
almost Gorenstein local ring. 

\begin{prop} \label{pg-agr}
Let $(A,\m)$ be a two-dimensional excellent normal local domain containing an algebraically closed field. 
Let $K_A$ denote the canonical module of $A$. 
If $\m$ is a $p_g$-ideal, then $A$ is an almost Gorenstein local ring in the sense of \cite{GTT}. 
That is, there exists a short exact sequence of $A$-modules$:$ 
\[
0 \to A \to K_A \to C \to 0
\]
such that $\m C=x C$ for some regular element $x$ over $C$.   
\end{prop}

\begin{exam} \label{Point-pg}
In the notation of Section 4, let $P\in C$ be such that 
$H_{C,P} = \{ 0, g+1,g+2,\ldots\}$.  
Then the maximal ideal $\m$ of 
$R_{C,P}$ is a $p_g$-ideal.    
\end{exam}

\begin{proof}  Take $f\in k(C)$ with $fT^{g+1} \in R$ and $v(f)= g+1$.
Then putting $Q=(T, fT^{g+1})$, we see that $\m^2=Q\m$. \par
Introduce a valuation $w$ on $R$ such that 
\[w(fT^m) =(m-v(f)) + \left[\dfrac{v(f)}{g+1} \right]\] 
To show that $\m^n$ is integrally closed it suffices to show that 
 $w(gT^m)\ge n$ if and only if $gT^m \in \m^n$.  If  $w(gT^m)\ge n$ and 
 $m - v(g) = r$, since $v(g) \ge (g+1)(n-r)$,  $g T^{v(g)}\in \m^{n-r}$ 
 and then  $gT^m \in \m^n$.\par
 Hence $\m^n$ is integrally closed for all $n\ge 1$ and $\m$ is a 
 $p_g$-ideal by   Proposition \ref{pg-chara}.
\end{proof}

In dimension $2$, although Takahashi-Dao's inequality does not hold for 
general normal ring $A$, we have an inequality adding a constant 
depending on $A$.  Also we have converse inequality for $e(I)$ changing 
$\LL(A)$ to $\ord(I)$.

\begin{prop}\label{Ineq_e(I)} 
Let $(A,\m)$ be an excellent normal
 local ring as in Theorem 
\ref{Dao} and $I$ be an integrally closed ideal in $A$. Then we have the 
following inequalities;
\begin{enumerate}
\item If we put $c = p_g(A) - \ell_A(H^1(X,\calO_X(-M)))$, then we have
 an inequality   
\[
(\mu(I)-1+c) \cdot \LL(I) \ge e(I).
\]
Note that $c$ is an invariant depending only on $A$.
\item If we change $\LL(I)$ to $\ord(I)$, we have the converse inequality;
\[
e(I) \ge (\mu(I) - 1) \cdot \ord(I).
\]
\end{enumerate}
\end{prop} 

\begin{proof}  
We modify our argument in the proof of Theorem \ref{Dao}.
\par
(1) In the situation of the proof of Theorem \ref{Dao}, 
by \cite[Theorem 6.1]{OWY1}, we have $\mu(I) \ge -MZ+1-c$ and 
then the argument is the same as Theorem \ref{Dao}.\par
(2) Since $I \subset \m^{\ord(I)}$, we have $Z \ge \ord(I)M$ and hence 
$-Z^2 \ge -\ord(I) MZ$ and $-MZ \ge \mu(I) -1$.  
\end{proof}

\begin{exam}  \label{Hyp-Ineq}
Put $A = k[[X,Y,Z]]/(f)$, where $f$ is homogeneous of degree 
$n\ge 3$ and we assume $A$ is normal. Then $c = 
p_g(A) - \ell_A(H^1(X,\calO_X(-M))) = \binom{n-1}{2}$
in this case.  In fact, if we take $I=\m^s$ with $s\ge n$, then 
we have $\LL(I) = s, \mu(I) = sn - (n-3)n/2$ and we see that $c = \binom{n-1}{2}$ is best possible.
\end{exam}

%\begin{rem}
%Shall we say something about CM version of $e(I) \ge %(\mu(I) - 1) \cdot \ord(I)$ ?
%\end{rem}

%\begin{quest} 
%\begin{rem}
In dimension $\ge 3$, 
%KW0721 there does not exist any constant $c$ 
we can take $A$ so that there is no constant $c$ 
for which the  inequality 
\[
(d-1)!(\mu(I)-d+1+c)\cdot \ell\ell(I) \ge e(I)
\]
hold for all integrally closed ideal $I$. 
See the next example. 
%\end{rem}
%\end{quest}

\begin{exam} \label{Hyp-const}
Let $A =k[[x,y,z,w]]/(f)$, where $f$ is a homogeneous polynomial 
of degree $n$ and we assume $A$ is normal. Then we can see if $n\ge 5$,
putting $I = \m^s$,  to have inequality 
$(d-1)!(\mu(I)-d+1+c)\cdot \ell\ell(I) \ge e(I)$ holds
we have to take $c$ arbitrary large when $s$ tends to infinity.
\end{exam}

%%%%%%%%%%%%%
\section{What Ideals have SRP ?}

We have shown under certain condition, 
$\m^n$ has SRP and also in regular local rings 
of dimension $2$, 
$\m^n$ ($n\ge 1$) are only ideals with SRP in 
Example \ref{RLR-SRP}. 
In this section, we ask if that this property 
characterize regular local rings and some Veronese subrings in dimension $2$.     
We get a partial result for rational singularities.

\begin{prop} \label{Only}
Let $(A,\m)$ be a two-dimensional rational singularity and 
assume that $A/\m$ is algebraically closed. 
Then we have the following 
results concerning the strong Rees property for integral closed ideals.
\begin{enumerate}
\item If an integrally closed ideal $I$ has the strong Rees property, then $I$ is a good ideal 
in the sense of \cite{GIW}.
\item If the minimal resolution of 
$\Spec(A)$ has more than two exceptional curves, 
then there are integrally closed ideals $I\ne \m^n$%KW0728
 with the property such that $\mu(I') < \mu(I)$ for every 
 integrally closed ideal $I'$ strictly containing $I$.  
%\item If $A$ is not regular, then there exists and $\m$-%primary ideal 
%$I\ne \m^n$ which has SRP.
\end{enumerate}
\end{prop}   

\begin{proof} If $A$ is a rational singularity, and if $I=I_Z$ 
 as in the proof of Theorem \ref{Dao}, we have shown 
$\mu(I) = - MZ +1$.  
\par
(1) If $I$ is not good and $I'$ is the minimal good ideal 
containing $I'$, we have $\mu(I) = \mu(I')$. Hence $I$ does not have SRP.
(2) If $I=I_Z$ has SRP, then $Z$ is defined on the minimal resolution of 
$\Spec(A)$ and the property \lq\lq $\mu(I') < \mu(I)$ for every integrally closed 
ideal $I'$ strictly containing $I$" is equivalent to say that \lq\lq for every cycle 
$Z' < Z$, $-MZ' < -MZ$ or %KW0721
$M(Z'-Z) >0$.  
If the minimal resolution of $\Spec(A)$
contains more than two curves, we can construct such 
$Z\ne nM$ with that property. 
\end{proof}

\begin{conj}  We believe that if $(A,\m)$ is a local ring of dimension 
$d \ge 2$ and if $\m^n$ are only ideals of $A$ which has the strong Rees property, then 
$\dim A=2$ and either $A$ is a regular local ring or $\hat{A} \cong 
k[[X^r,X^{r-1}Y, \ldots , Y^r]]$ for some $r$.
\end{conj}

\begin{ac}
The authors would like to thank Hailong Dao 
for giving us a motivation of this work. 
Moreover, they would like to thank Shiro Goto, J\"urgen Herzog, Ryo Takahashi, Junzo Watanabe and Santiago Zarzuela for giving us several valuable comments.  
\par 
The second author is supported by Grant-in-Aid for Scientific Research (C) 26400053.
The third author is supported by Grant-in-Aid for Scientific Research (C) 16K05110.
\end{ac}

%%%%%%%%%%%%%%

%%%%%%%%%%%%%%%%%%%%%%%%%%%%%%%%%%%%%%%%%%%


\begin{thebibliography}{20}

\bibitem{AP}
J.~Asadollahi and T.~J.~Puthenpurakal, 
\emph{An analogue of a theorem due to Levin and Vasconcelos}, 
Commutative algebra and algebraic geometry, 9--15,
Contemp. Math., {\bf 390}, Amer. Math. Soc., Providence, RI, 2005. 

\bibitem{DS}
H.~Dao and I. ~Smirnov, 
\textit{The multiplicity and the number of generators of an integrally closed ideal}, arXiv:1703.09427.

\bibitem{G}
S.~Goto, 
\textit{Integral closedness of complete intersection ideals}, 
J.~Algebra {\bf 108} (1987), 151--160.

\bibitem{GIW}
S. ~Goto, S.~Iai, and K.-i. ~Watanabe, 
\textit{Good ideals in  {G}orenstein local rings}, 
Trans. Amer. Math. Soc. \textbf{353} (2001),
  no.~6, 2309--2346 (electronic).


%\bibitem{GS}
%S. ~Goto and Y. ~Shimoda, \textit{On the {R}ees algebras of
 % {C}ohen-{M}acaulay local rings}, Commutative algebra ({F}airfax, %{V}a.,
%  1979), Lecture Notes in Pure and Appl. Math., vol.~\textbf{68}, %Dekker, New York,
%  1982, pp.~201--231.% \MR{655805 (84a:13021)}

\bibitem{GW} S.~Goto and K. -i.Watanabe, \emph{On graded rings, I, } 
 J. Math. Soc. Japan, {\bf 30} (1978), 179--213.  
 
\bibitem{GTT}
S.~Goto, R.~Takahashi and N.~Taniguchi, 
\textit{Almost Gorenstein rings--towards a theory of 
higher dimension}, 
J. Pure Appl. Algebra {\bf 219} (2015), 2666--2712.

\bibitem{He} J. Herzog, 
\textit{When is a regular sequence super regular?}, 
Nagoya Math.~J. {\bf 83} (1981), 183--195.  

\bibitem{HH}
S.~Huckaba and C.~Huneke, 
\emph{Normal ideals in regular rings}, 
J. Reine Angew. Math. 510 (1999), 63-82. 
%KW0728  Checked MathSciNet
\bibitem{I}
S.~Itoh,
\emph{Coefficients of normal Hilbert polynomials},
J. Algebra {\bf 150} (1992), no.1, 101--117. 

\bibitem{Ko} J. ~Komeda, \textit{On the existence of weierstrass points with a certain semigroup generated by 4 elements}, 
Tsukuba J.~ Math. {\bf 6}  (1982), no.2, 237--270.

\bibitem{K-O} J.~Komeda, and A.~Ohbuchi,
 \textit{Existence of the non-primitive Weierstrass gap sequences on curves of genus 8}, 
Bull Braz Math. Soc. News Series {\bf 39} (2008), 109--121.

%KW0529
\bibitem{LV}
L. ~Robbiano and G. ~Valla, 
\textit{On the equations defining tangent cones}, 
Math. ~Proc. ~Camb. ~Phil. ~Soc.  {\bf 88}  (1980),  281--297.

\bibitem{OWY1} 
T.~Okuma, K.-i.~Watanabe and K.~Yoshida, 
\emph{Good ideals and $p_g$-ideals in two-dimensional normal singularities}, 
manuscripta math. {\bf 150} (2016), 499--520. 

\bibitem{OWY2} 
T.~Okuma, K.-i.~Watanabe and K.~Yoshida, 
\emph{Rees algebras and $p_g$-ideals in a two-dimensional normal local domain}, 
Proc. Amer. Math. Soc. {\bf 145}  (2017), no.1, 39--47. 


\bibitem{P1}
T.~J.~Puthenpurakal, 
\emph{Ratliff-Rush filtration, regularity and depth of higher associated graded modules. I},
J. Pure Appl. Algebra {\bf 208}  (2007), no.1, 159--176. 

\bibitem{P2}
T.~J.~Puthenpurakal, 
\emph{Ratliff-Rush filtration, regularity and depth of higher associated graded modules. Part II},
 J. Pure Appl. Algebra \textbf{221} (2017), no. 3, 611--631.

\bibitem{RR}
L.~J.~Ratliff, Jr. and D.~E.Rush, 
\emph{Two Notes on Reduction of ideals}, 
Indiana Univ. Math. J. \textbf{27}  (1978), no.6, 929--934.   

\bibitem{Sa1}
J.D.~Sally, 
\emph{On the associated graded ring of a local Cohen-Macaulay ring}, 
J.~Math. Kyoto Univ. {\bf 17} (1977), 19--21. 

\bibitem{Sa2}
J.D.~Sally, 
\emph{Cohen-Macaulay local rings of maximal 
embedding dimension}, 
J.~Algebra {\bf 56} (1979), 168--183. 

\bibitem{Sa3}
J.D.~Sally, 
\emph{Cohen-Macaulay Local Rings of Embedding Dimension $e+d-2$}, 
J.~Algebra {\bf 83} (1983), 393--408. 

\bibitem{DT}
H.~Dao and R.~Takahashi, 
\emph{Upper bounds for dimensions of singularity categories}, 
C. R. Acad. Sci. Paris, Ser. I, 
{\bf 353} (2015),  297--301. 

\bibitem{JW}
 J.~Watanabe, 
\emph{$\m$-full ideals}, 
  Nagoya Math. ~J. {\bf 106} (1987), 101--111.  

\bibitem{KW}
 K.-i. ~Watanabe,
 \emph{Chains of integrally closed ideals},
  Commutative algebra (Grenoble/Lyon, 2001), 353--358,
  Contemp. Math., {\bf 331}, Amer. Math. Soc., Providence, RI, 2003.
\end{thebibliography}
\end{document}